\documentclass[reqno]{amsart}

\usepackage{hyperref}
\UseRawInputEncoding
\usepackage{amscd,amssymb,amsmath,verbatim}
\newtheorem {thm}{Theorem}[section]
\newtheorem{lem}[thm]{Lemma}
\newtheorem{prop}[thm]{Proposition}
\newtheorem{cor}[thm]{Corollary}
\newtheorem{df}[thm]{Definition}

\newtheorem{ex}[thm]{Example}
\newtheorem{exs}[thm]{Examples}

\begin{document}

\title{Duo property for rings by the quasinilpotent perspective }

\author{Abdullah Harmanci}
\address{Abdullah Harmanci, Department of Mathematics, Hacettepe
University, Ankara,~ Turkey}\email{harmanci@hacettepe.edu.tr}

\author{Yosum Kurtulmaz}
\address{Yosum Kurtulmaz, Department of Mathematics, Bilkent University,
Ankara, Turkey}\email{yosum@fen.bilkent.edu.tr}

\author{Burcu Ungor}
\address{Burcu Ungor, Department of Mathematics, Ankara
University, Ankara, Turkey}\email{bungor@science.ankara.edu.tr}

\date{}

\begin{abstract} In this paper, we focus on the duo ring property
via quasinilpotent elements which gives a new kind of
generalizations of commutativity. We call this kind of ring {\it
qnil-duo}. Firstly, some properties of quasinilpotents in a ring
are provided. Then the set of quasinilpotents is applied to the
duo property of rings, in this perspective, we introduce and study
right (resp., left) qnil-duo rings. We show that this concept is
not left-right symmetric. Among others it is proved that if the
Hurwitz series ring $H(R; \alpha)$ is right qnil-duo, then $R$ is
right qnil-duo. Every right qnil-duo ring is abelian. A right
qnil-duo exchange ring has stable range 1.
 \vskip 0.5cm

\noindent {\bf 2010 MSC:}  16U80, 16U60, 16N40

\noindent {\bf Keywords:} Quasinilpotent element, duo ring,
qnil-duo ring
\end{abstract}
\maketitle
\section{Introduction}
Throughout this paper, all rings are associative with identity.
Let $N(R)$, $J(R)$, $U(R)$, $C(R)$ and Id$(R)$ denote the set of
all nilpotent elements, the Jacobson radical,  the set of all
invertible elements, the center and the set of  all idempotents of
a ring $R$, respectively. We denote the $n\times n$ full (resp.,
upper triangular) matrix ring over $R$ by $M_n(R)$ (resp.,
$U_n(R))$, and $D_n(R)$ stands for the subring of $U_n(R)$
consisting of  all matrices which have equal diagonal entries and
$V_n(R) = \{(a_{ij})\in D_n(R)\mid a_{ij} = a_{(i+1)(j+1)}$ for $i
= 1,\dots,n - 2$ and $j = 2,\dots,n-1\}$ is a subring of $D_n(R)$.
Let $\Bbb Z$ and $\Bbb Z_n$ denote the ring of integers and the
ring of integers modulo $n$ where $n\geq 2$.

In \cite{F}, Feller introduced the notion of duo rings, that is, a
ring is called {\it right} (resp., {\it left}) {\it duo} if every
right (resp., left) ideal is an ideal, in other words,
$Ra\subseteq aR$ (resp., $aR\subseteq Ra$) for every $a\in R$, and
a ring is said to be {\it duo} if it is both right and left duo. The duo
ring property was studied in different aspects. For example, in
\cite {HLP}, the concept of right unit-duo ring was introduced,
namely, a ring $R$ is called {\it right unit-duo} if for every
$a\in R$, $U(R)a\subseteq aU(R)$. Left unit-duo rings are defined
similarly. In \cite{KY}, the normal property of elements on
Jacobson and nil radicals were concerned. A ring $R$ is called
{\it right normal on Jacobson radical} if $J(R)a\subseteq aJ(R)$
for all $a\in R$. Left normal on Jacobson radical rings can be
defined analogously. Also in \cite{KY}, on the one hand, a ring
$R$ is said to satisfy the {\it right normal on upper nilradical}
if $N^*(R)a\subseteq aN^*(R)$ for all $a\in R$ where $N^*(R)$ is
the upper nilradical of $R$. Similarly, left normal on upper
nilradical rings are defined similarly. On the other hand, a ring
$R$ is said to satisfy the {\it right normal on lower nilradical}
if $N_*(R)a\subseteq aN_*(R)$ for all $a\in R$ where $N_*(R)$ is
the lower nilradical of $R$. Similarly, left normal on lower
nilradical rings are defined similarly. Also, a ring $R$ is called
{\it right nilpotent-duo} if $N(R)a \subseteq aN(R)$ for every
$a\in R$. Left nilpotent duo rings are defined similarly (see
\cite{HKKKL}).

Motivated by the works on duo property for rings, the goal of this
paper is to approach the notion of duo rings by the way of
quasinilpotent elements, in this regard, we introduce the notion of
qnil-duo rings. Firstly, we investigate some properties of
quasinilpotent elements which we need for the investigation of
qnil-duo property. Then we study some properties of this class of
rings and observe that being a qnil-duo ring need not be
left-right symmetric. It is proved that any right (resp., left)
qnil-duo ring is abelian, and any exchange right (resp., left)
qnil-duo ring has stable range 1. It is observed that regularity
and strongly regularity coincide for right (resp., left) qnil-duo
rings.  We also study on some extensions of rings such as Dorroh
extensions, Hurwitz series rings and some subrings of matrix rings
in terms of qnil-duo property.

\section{Some properties of quasinilpotents}
\noindent Let $R$ be a ring and $a\in R$. The {\it commutant} and
{\it double commutant} of $a$ in $R$ are defined by comm$(a) =
\{b\in R\mid ab = ba\}$ and comm$^2(a) = \{b\in R\mid bc = cb $
for all $c\in ~\mbox{comm}(a)\}$, respectively, and $R^{qnil} =
\{a\in R \mid 1 + ax$ is invertible in $R$ for every $x \in $
comm$(a)\}$. Elements of the set $R^{qnil}$ are called {\it
quasinilpotent} (see \cite{Ha}). Note that $J(R) = \{a\in R\mid 1
+ ax$ is invertible for $x\in R\}$. If $a\in N(R)$ and $x\in $
comm$(a)$, then $ax\in N(R)$ and $1 + ax\in U(R)$. So
$J(R)\subseteq R^{qnil}$, $N(R)\subseteq R^{qnil}$ and $R^{qnil}$
does not contain invertible elements, $0\in R^{qnil}$ but the
identity is not in $R^{qnil}$. In this section, we start to expose some properties of $R^{nil}$ and continue to study some other properties
of quasinilpotent elements in rings.
\begin{ex} {\rm There are rings $R$ such that $J(R)$ is strictly contained in $R^{qnil}$.}
\end{ex}
\begin{proof} Let $F$ be a field and $R =  M_n(F)$ for some positive integer $n$. Then $J(R) = 0$ and the matrix unit $E_{1n}$ belongs to $R^{qnil}$ but not $J(R)$.
\end{proof}

\noindent We now mention some of the known facts about
quasinilpotents for an easy reference.
\begin{prop}\label{nilö} \begin{enumerate} \item Let $R$ be a ring, $a\in R$ and $n$ a positive integer.
If $a^n\in R^{qnil}$, then $a\in R^{qnil}$, in particular every
nilpotent element is in $R^{qnil}$ {\rm (}\cite[Proposition
2.7]{Cu}{\rm)}. \item  If $R$ is a local ring, then $U(R)\cap
R^{qnil} = \emptyset$ and $R = U(R)\cup R^{qnil}$
{\rm(}\cite[Theorem 3.2]{Cu}{\rm )}.\item  Let $R$ be a ring, $a$,
$b\in R$. Then $ab\in R^{qnil}$ if and only if $ba\in R^{qnil}$
{\rm(}\cite[Lemma 2.2]{LZ}{\rm)}. \item  Let $a\in R^{qnil}$ and
$r\in U(R)$. Then $r^{-1}ar\in R^{qnil}$ {\rm (}\cite[Lemma
2.3]{Cu}{\rm)}. \item Let $e^2=e\in(R)$. Then $(eRe)^{qnil} =
(eRe)\cap R^{qnil}$ {\rm(}\cite[Lemma 3.5]{ZC}{\rm)}.
\end{enumerate}
\end{prop}

\noindent In the following, we determine quasinilpotent elements
in some classes of rings.
\begin{lem}\label{nil} Let $R$ be a ring. Then the following hold.
\begin{enumerate} \item  $\left\{\begin{bmatrix}a&b\\0&c\end{bmatrix}\mid a, c\in R^{qnil}, b\in R\right\}
\subseteq U_2(R)^{qnil}$.
\item  $\left\{\begin{bmatrix}a&b\\0&a\end{bmatrix}\mid a\in R^{qnil}, b\in R\right\}\subseteq D_2(R)^{qnil}$.
\item  Let $A = \begin{bmatrix}a&b\\0&a\end{bmatrix}\in D_2(R)^{qnil}$ with $b\in $ comm$^2(a)$. Then $a\in R^{qnil}$.
\end{enumerate}
\end{lem}
\begin{proof} (1) Let $a$, $c\in R^{qnil}$, $b\in R$,
$A = \begin{bmatrix}a&b\\0&c\end{bmatrix}\in U_2(R)$ and $B =
\begin{bmatrix}x&y\\0&z\end{bmatrix}\in $ comm$(A)$. Then $AB = BA$
implies $1 - ax$ and $1 - cz$ are invertible. Hence $I - AB$ is
invertible. So $A\in R^{qnil}$.  \\(2) Let $A =
\begin{bmatrix}a&b\\0&a\end{bmatrix}\in D_2(R)$ with $a\in
R^{qnil}$, $b\in R$ and $B =
\begin{bmatrix}x&y\\0&x\end{bmatrix}\in $ comm$(A)$. Then $AB = BA$
implies $x\in $ comm$(a)$. Then $1 - ax$ is invertible. Hence $I_2
- AB$ is invertible. So $A\in D_2(R)^{qnil}$.\\(3) Clear.
\end{proof}

\begin{prop}\label{ilk} Let $(R_i)_{i\in I}$ be a family of rings for some index set $I$ and
let $R = \prod_{i\in I}R_i$. Then $R^{qnil} = \prod_{i\in
I}R^{qnil}_i$.
\end{prop}
\begin{proof} Let $(a_i)$, $(x_i)\in R$. Then $(x_i)\in$ comm$(a_i)$ if and only if $x_i\in$ comm$(a_i)$ for
all $i\in I$. Hence $1 + (a_i)(x_i)$ is invertible in $R$ if and only if
$1 + a_ix_i$ is invertible in $R_i$ for every $i\in I$. So the
result follows.
\end{proof}

\noindent Let $R$ be  an algebra over a commutative ring $S$. The
{\it Dorroh extension} (or {\it ideal extension}) of $R$ by $S$
denoted by $I(R, S)$ is the direct product $R\times S$ with usual
addition and  multiplication defined by $(a_1, b_1)(a_2, b_2) =
(a_1a_2 + b_1a_2 + b_2a_1, b_1b_2)$ for $a_1$, $a_2\in R$ and
$b_1$, $b_2\in S$.
\begin{lem}\label{IR} Let $I(R, S)$ be an ideal extension. Then the following hold. \begin{enumerate}
\item [(1)] For $(a, b)\in I(R, S)$, $(c, d)\in $ comm$(a, b)$ if and only if $c\in $ comm$(a)$.
\item [(2)] $(a, b)$ has an inverse $(c, d)$ in $I(R, S)$ if and only if $(a + b)(c + d) = 1 = (c + d)(a + b)$ and $bd = db = 1$.
\end{enumerate}
\end{lem}
\begin{proof} (1) $(c, d)\in $ comm$(a, b)$ if and only if $(a, b)(c, d) = (c, d)(a, b)$ if and only if $ac + da + bc = ca + da + bc$ and $bd = db$ if and only if $ac = ca$ and $bd = db$ if and only if  $c\in $ comm$(a)$.\\
(2) $(a, b)(c, d) = (0, 1)= (c, d)(a, b)$ if and only if $ac + da
+ bc = ca + da + bc = 0$ and $bd = db = 1$ if and
 only if $ac + da + bc + bd + (-bd) = (a + b)(c + d) - 1 = 0$ and $bd = db = 1$.
\end{proof}
\begin{prop}\label{nil2} Let $I(R, S)$ be an ideal extension of $R$ by $S$. Then
\begin{enumerate}
\item [(1)] $(R, 0)^{qnil} = (R, 0)\cap I(R, S)^{qnil}$.
\item [(2)] $(0, S)\cap I(R, S)^{qnil}\subseteq (0, S)^{qnil}$.
\end{enumerate}
\end{prop}
\begin{proof} (1) Let $(x, 0)\in (R, 0)^{qnil}$ and $(a, b)\in I(R, S)$
with $(a, b)\in $ comm$(x, 0)$. Then $a\in $ comm$(x)$ and so $1 +
xa$ is invertible in $R$. We prove $(0, 1) + (x, 0)(a, b)$ is
invertible. Since $S$ lies in the center of $R$, $a + b\in $
comm$(x)$. Hence $1 + x(a + b)$ is invertible, say $(1 + x(a +
b))(u + 1) = (u + 1)(1 + x(a + b)) = 1$. This implies that $u(x(a
+ b)) + x(a + b) + u = 0$. Hence  $((0, 1) + (x, 0)(a, b))(u, 1) =
(u, 1)((0, 1) + (x, 0)(a, b)) = (0, 1)$ for all $(a, b)\in $
comm$(x, 0)$. So $(x, 0)\in I(R, S)^{qnil}$.

Conversely, let $(x, 0)\in (R, 0)\cap I(R, S)^{qnil}$ and $(r,
0)\in $ comm$(x, 0)$. Hence $(0, 1) + (x, 0)(r, 0)= (rx, 1)$ is
invertible. Let $(a, b)$ be the inverse of $(rx, 1)$. Then $(rx,
1)(a, b) = (0, 1)$ implies $b = 1$ and $rxa + rx + a = 0$. $(a, 1)
(rx, 1) = (0, 1)$ implies $arx + a + rx = 0$. Hence $(1 + a)(1 +
rx) = 1$ and $(1 + rx)(1 + a) = 1$.
Hence $(1, 0) + (r, 0)(x, 0)$ is invertible in $(R, 0)$ for all $(r, 0)\in $ comm$(x, 0)$. Thus $(x, 0)\in (R,0)^{qnil}$ or $(R, 0)\cap I(R, S)^{qnil}\subseteq (R, 0)^{qnil}$.\\
(2) Let $(0, s)\in (0, S)\cap I(R, S)^{qnil}$. Let $(0, b)\in (0,
S)$ with $(0, b)\in $ comm$(0, s)$. Then $(0, 1) + (0, s)(0, b)=
(0, 1 + sb)$ is invertible in $I(R, S)$. There exists $(u, v)\in
I(R, S)$ such that $(0, 1 + sb)(u, v) = ((1 + sb)u, (1 + sb)v) =
(0, 1) = (u, v)(0, 1 + sb) = ((1 + sb)u, v(1 + sb))$. Hence $(1 +
sb)v = v(1 + sb) = 1$ and $(1 + sb)u = 0$. Hence $u = 0$. Thus
$(0, 1) + (0, s)(0, b)= (0, 1 + sb)$ is invertible in $(0, S)$
with inverse $(0, v)\in (0, S)$. It follows that $(0, s)\in (0,
S)^{qnil}$ and so $(0,S)\cap I(R, S)^{qnil}\subseteq (0,
S)^{qnil}$.
\end{proof}

\noindent The following gives us  necessary and sufficient
conditions for $(0, S)^{qnil}$ to be contained in $I(R,
S)^{qnil}$.
\begin{thm}\label{IRS} Let $I(R, S)$ be the ideal extension of an algebra $R$ by a commutative
 ring $S$. Let $(0, i)\in (0, S)^{qnil}$. Then $(0, i)\in I(R, S)^{qnil}$ if and only if for every
 $(a, b)\in $ comm$(0, i)$ there exists $(u, v)\in I(R, S)$ such that $(i(a + b) + 1)(u + v) = (1 + ib)v = 1$.
\end{thm}
\begin{proof} Assume that  $(0, i)\in I(R, S)^{qnil}$.
Let $(a, b)\in $ comm$(0, i)$ in $I(R, S)$. Then $(0, 1) + (0,
i)(a, b)$ must be invertible. There exists $(u, v)\in I(R, S)$
such that $(0, 1) = ((0, 1) + (0, i)(a, b))(u, v)$. It follows
that $(0, 1) = ((0, 1) + (0, i)(a, b))(u, v) = (ia, 1 + ib)(u, v)
= (iau + (1 + ib)u + v(ia), (1 + ib)v)$. Then $iau + (1 + ib)u +
iav = 0$ and $(1 + ib)v = 1$. They lead us $((i(a + b) + 1)(u + v)
= (1 + ib)v$. Hence $(i(a + b) + 1)(u + v)$ is invertible.
Conversely, note that $iau + (1 + ib)u + iav = 0$ if and only if
$((i(a + b) + 1)(u + v) = (1 + ib)v$. Assume that for $(0, i)\in
(0, S)$ there exists $(u, v)\in I(R, S)$ such that  $(i(a + b) +
1)(u + v) = (1 + ib)v = 1$. Then by concealing paranthesis we may
reach that $(0, 1) =  ((0, 1)+ (0, i)(a, b))(u, v)$ for $(a, b)\in
$ comm$(0, i)$. Hence $(0, i)\in I(R, S)^{qnil}$.
\end{proof}

\noindent Let $R$ be a ring and $S$ a subring of $R$ with the same identity as that of $R$ and
$$T[R, S] = \{ (r_1, r_2, \cdots , r_n, s, s, \cdots) : r_i\in R, s\in S, n\geq 1, 1\leq i\leq n\}.$$ Then $T[R, S]$ is
a ring under the componentwise addition and multiplication. Note
that N$(T[R,S])= T[N(R), N(S)]$ and $C([T,S]) = T[C(R), C(R)\cap
C(S)]$.
\begin{prop}\label{TRS} Let $R$ be a ring and $S$ a subring of $R$ with the same identity as that of $R$.
\begin{enumerate}
\item If $A = (a_1, a_2, a_3, \dots, a_n, s, s, s, \dots)\in T[R, S]^{qnil}$, then $a_i\in R^{qnil}$ for $i = 1, 2, 3, \dots, n$ and $s\in S^{qnil}$.
\item If $a\in R^{qnil}$ and $s\in S^{qnil}$, then $A = (a, s, s, s, \dots,..)\in T[R,
S]^{qnil}$.
\end{enumerate}
\end{prop}
\begin{proof} (1) Let $A = (a_1, a_2, a_3, \dots, a_n, s, s, s, \dots)\in T[R, S]^{qnil}$
and $b_i\in $ comm$(a_i)$ and $t\in $ comm$(s)$. Then $B = (b_1,
b_2, b_3, \dots, b_n, t, t, t, \dots)\in $ comm$(A)$. Let $1 = (1,
1, 1, \dots, 1, \dots)$ denote the identity of $ T[R, S]$. So $1 +
AB$ is invertible. Therefore $1 + a_ib_i$ is invertible for $i =
1, 2, 3, \dots, n$ and $1+st$ is invertible in $S$. Hence $a_i\in
R^{qnil}$ for $i = 1, 2, 3, \dots, n$ and $s\in S^{qnil}$.\\(2)
Let $a\in R^{qnil}$, $s\in S^{qnil}$, $A = (a, s, s, \dots)$. If
$B = (b_1, b_2, \dots, b_m, t, t, t, \dots.)\in T[R, S]$ lies in
comm$(A)$, then $b_1\in $ comm$(a)$, $b_i\in $ comm$(s)$  for $i =
2,3 ,\dots, m$ and $t\in $ comm$(s)$. Hence $1 + ab_1$ and  $1 +
sb_i$ are invertible in $R$ where $i = 2,3 ,\dots, m$ and $1 + st$
is invertible in $S$. Hence $1 + AB$ is invertible in $T[R, S]$.
So $A = (a, s, s, \dots)\in T[R, S]^{qnil}$.
\end{proof}
%\begin{prop}\label{prod}\cite{KUH} Let $(R_i)_{i\in I}$ be a family of rings for some index set $I$. Then $(\prod_{i\in I}R_i)^{qnil} = \prod_{i\in I}R_i^{qnil}$.
%\end{prop}
 %A ring R has stable range 1 if aR + bR = R with a,b ∈ R there exists a y ∈ R such that a + by ∈ R is invertible.
%????Lemma 2.4 If R has stable range 1, then J(R) = {x | x−u ∈ U(R) for any u ∈ U(R)}. Proof Let x ∈ R and x−u ∈ U(R) for any u ∈ U(R). Let r ∈ R. Then xR +(1−xr)R = R. Since R has %stable range 1, we have a y ∈ R such that u := x + (1−xr)y ∈ U(R). Hence, x−u = −(1−xr)y ∈ U(R), and then 1−xr ∈ U(R). Therefore, x ∈ J(R), and thus yielding the result.((Combining %Euclidean and adequate rings
%Huanyin CHEN1,∗, Marjan SHEIBANI2 TJM,Turk J Math (2016) 40: 506 – 516 ))???
\noindent Let $R$ be a ring with an endomorphism $\alpha$ and let
$H(R; \alpha)$ be the set of formal expressions of the type $f(x)
= \sum^{\infty}_{n=0}a_nx^n$ where $a_n\in R$ for all $n\geq 0$.
Define addition as componentwise and $*$-product on $H(R; \alpha)$
as follows: for $f(x) = \sum^{\infty}_{n=0}a_nx^n$ and $g(x) =
\sum^{\infty}_{n=0}b_nx^n$, $f*g  = \sum^{\infty}_{n=0}c_nx^n$
where $c_n = \sum^n_{n=0}\binom{n}ia_ib_{n-i}$. Then $H(R;
\alpha)$ becomes a ring with identity containing $R$ under these
two operations. The ring $H(R; \alpha)$ is called the {\it Hurwitz
series ring} over $R$. The {\it Hurwitz polynomial ring} $h(R;
\alpha)$ is the subring of $H(R; \alpha)$ consisting of formal
expressions of the form $\sum^n_{i=0}\binom{n}ia_ix^i$. Let
$\epsilon : H(R; \alpha)\rightarrow R$ defined by $\epsilon(f(x))
= a_0$. Then $\epsilon$ is a homomorphism with ker$(\epsilon)=
xH(R; \alpha)$ and $H(R; \alpha)/$ker$(\epsilon)\cong R$. There
exist one-to-one correspondences between $H(R; \alpha)$ and $R$
relating to invertible elements, commutants and ideals. Let
$R[[x;\alpha]]$ be the skew formal power series ring over $R$. The
sum is the same but multiplication in $H(R, \alpha)$ is similar to
the usual multiplication of $R[[x;\alpha]]$, except that binomial
coefficients appear in each term in the multiplication defined in
$H(R, \alpha)$. Also, there is a ring homomorphism $\epsilon$
between $R[[x; \alpha]]$ and $R$, defined by $\epsilon(f(x)) =
a_0$ where $f(x) = a_0 + a_1x + a_2x^2 + \cdots\in R[[x;
\alpha]]$. Clearly, $\epsilon$ is an onto map and  $R[[x;
\alpha]]/$ker$(\epsilon)\cong R$.

\begin{lem} Let $R$ be a ring and $\alpha$ a ring endomorphism of $R$. Then
\begin{enumerate}\item $U(H(R; \alpha)) = \epsilon^{-1}U(R)$,\item $U(R[[x; \alpha]]) = \epsilon^{-1}U(R)$.\end{enumerate}
\end{lem}
\begin{proof} It is routine.
\end{proof}

\noindent In the next result, we determine the quasinilpotent elements of
$H(R; \alpha)$ and $R[[x; \alpha]]$.
\begin{prop}\label{HR} \begin{enumerate}\item[(1)] Let $H(R; \alpha)$ be a Skew Hurwitz series ring over $R$.
Then $H(R; \alpha)^{qnil} = \epsilon^{-1}R^{qnil}$.\item[(2)] Let
$R[[x; \alpha]]$ be a Skew formal power series ring over $R$. Then
$R([[x; \alpha]])^{qnil} = \epsilon^{-1}R^{qnil}$. \end{enumerate}
\end{prop}
\begin{proof} (1) Let $f(x) = a_0 + a_1x + a_2x^2 + \dots\in H(R; \alpha)^{qnil}$
and $r\in R$ with $r\in$ comm$_R(a_0)$. Then $r\in$ comm$_{H(R;
\alpha)}(a)$. Then $1 + f(x)r\in U(H(R; \alpha))$. Hence $1 +
a_0r\in U(R)$. So $a_0\in R^{qnil}$. Since $a_0 = \epsilon(f(x))$,
$f(x)\in \epsilon^{-1}(R^{qnil})$. Conversely, let $g(x) = b_0 +
b_1x + b_2x^2 + \dots\in \epsilon^{-1}(R^{qnil})$. Then
$\epsilon(g(x)) = b_0\in R^{qnil}$. Let $h(x) = c_0 + c_1x +
c_2x^2 + \dots\in$ comm$_{H(R,\alpha)}(g(x))$. Then $c_0\in$
comm$_R(b_0)$ and  so $1 + b_0c_0\in U(R)$. So $1 + g(x)h(x)\in
U(H(R,\alpha))$. Hence $g(x)\in H(R,\alpha)^{qnil}$. This
completes the proof.\\(2) Similar that of (1).
\end{proof}

\section{Qnil-duo Rings}

\noindent In this section, we deal with the right duo property on the set of
quasinilpotent elements. By this means, we give a generalization
of commutativity from the perspective of quasinilpotents.

\begin{df}{\rm  A ring $R$ is called {\it right qnil-duo}
if $R^{qnil}a\subseteq aR^{qnil}$ for every $a\in R$. Similarly,
$R$ is called {\it left qnil-duo}  if $aR^{qnil}\subseteq
R^{qnil}a$ for every $a\in R$. If $R$ is both right and left
qnil-duo, then it is called {\it qnil-duo}, i.e. $R^{qnil}a=
aR^{qnil} $ for every $a\in R$.}
\end{df}

\noindent The qnil-duo property of rings  is not left-right symmetric as the
following example shows.
\begin{ex}{\rm Let $S = F(t)$ denote the quotient field of the polynomial ring
$F[t]$ over a field $F$ and $\alpha\colon S\rightarrow S$ defined
by $\alpha(f(t)/g(t)) = f(t^2)/g(t^2)$. Let $R = S[[x; \alpha]]$
denote the skew power series ring with $xa = \alpha(a)x$ for $a\in
S$. Every element of $R$ is of the form $a =
\sum^{\infty}_{i=0}a_ix^i$. For any $r = a_0 +
\sum^{\infty}_{i=1}a_ix^i$ with $a_0\neq 0$ is invertible. Hence
$R^{qnil} = xR$. This ring is considered in  \cite[Lemma
1.3(3)]{CL}, \cite[Example 1]{KKL} and in \cite[Example
1.5]{HKKKL}.  As in the proof of \cite[Example 1]{KKL}, for
$tx^m\in tR^{qnil}$, there is no $g(x)\in R^{qnil}$ such that
$tx^m = g(x)t$. Hence $R$ is not left qnil-duo. We show that $R$ is
right qnil-duo. Let $f(x)\in R^{qnil}$, $g(x)\in R$. We show that there
exists $f_1(x)\in R^{qnil}$ such that $f(x)g(x) = g(x)f_1(x)$.
Assume that $g(x)$ is invertible. Then $f(x)g(x) =
g(x)(g(x)^{-1}f(x)g(x)\in g(x)R^{qnil}$, otherwise, let $g(x) =
h(x)x^m$ where $h(x) = a_0 + a_1x^+ a_2x^2 +\cdots$ is invertible.
Then $f(x)g(x) = f(x)h(x)x^m = f(x)x^mh_1(x) = x^mf_1(x)h_1(x)=
x^mh_1(x)(h_1(x)^{-1}f_1(x)h_1(x)) =
g(x)(h_1(x)^{-1}f_1(x)h_1(x))\in g(x)R^{qnil}$ since $f(x)$ is not
invertible and $f_1(x)$ is an application of $x^m$ to $f(x)$ from
the right, therefore $f_1(x)=x^kf_2(x)\in R^{qnil}$ for some
$k\geq 1$, by Proposition \ref{nilö}(4),
$h_1(x)^{-1}f_1(x)h_1(x)\in R^{qnil}$. Thus $R$ is right qnil-duo.
}
\end{ex}

\begin{exs}\label{ör1} {\rm (1) All commutative rings, all division rings are qnil-duo.\\
(2) There are local rings that are not right qnil-duo. }
\end{exs}
\begin{proof} (1) When $R$ is a commutative ring, it is both right and left qnil-duo.
If $R$ is a division ring, then $R^{qnil}=\{0\}$, therefore $R$ is
both right and left qnil-duo.\\(2) Let $A= \Bbb Z_4[x, y]$ be the
polynomial ring with non-commuting indeterminates $x$ and $y$ and
$I$ be the ideal generated by the set $\{x^3, y^2, yx, x^2 - xy,
x^2-2, 2x, 2y\}$. Consider the ring $R = A/I$. By \cite[Example
7]{XX}, $R$ is a local ring. It is easily checked that $R^{qnil} =
\{0, 2, x, y, 2+x, 2+y, 2 + x + y, x + y\}$ and $(R^{qnil})^2 \neq
0$. $2 + x$ belongs to $R^{qnil}$ since it is nilpotent. For $x\in
R^{qnil}$ and $y\in R$, $xy\in R^{qnil}y$. It is easily checked
that there is no $t\in R^{qnil}$ such that $xy = yt\in yR^{qnil}$.
Hence $R$ is not right qnil-duo.
\end{proof}
\begin{lem}  Let $R$ be a ring with $R^{qnil}$ central in $R$. Then $R$ is  qnil-duo.
\end{lem}
\begin{proof}  Assume that $R^{qnil}$ is central in $R$. Let $a\in R$ and $b\in R^{qnil}$. Then $b$ being central implies $ab = ba\in aR^{qnil}$.
\end{proof}
\begin{thm}\label{dik}  Let $\{R_i\}_{i\in I}$ be a family of rings for some index set $I$ and $R = \prod_{i\in I}R_i$. Then $R_i$ is right (resp., left) qnil-duo for each $i\in I$ if and only if $R$ is right (resp., left) qnil-duo.
\end{thm}
\begin{proof} Assume that $R_i$ is right (resp., left) qnil-duo for each $i\in I$. Let $a = (a_i)\in R$, $b=(b_i)\in R^{qnil}$.
By Proposition \ref{ilk}, $b_i\in R_i^{qnil}$ for each $i\in I$.
By assumption there exists $c_i\in R_i^{qnil}$ such that $b_ia_i =
a_ic_i$ for each $i\in I$. Set $c = (c_i)$. Then $ba = ac\in
aR^{qnil}$. Hence $R^{qnil}a\subseteq aR^{qnil}$. Conversely,
suppose that $R$ is right qnil-duo. Let $a_i\in R_i$ and $b_i\in
R_i^{qnil}$ where $i\in I$. Let $a = (a_i)$, $b = (b_i)\in R$
where $i^{th}$-entry of $a$ is $a_i$ and the other entries are 0
and $i^{th}$-entry of $b$ is $b_i$ and the other entries are 0,
respectively. Then $a = (a_i)\in R$ and by Proposition \ref{ilk},
$b\in R^{qnil}$. Supposition implies there exists $c = (c_i)\in
R^{qnil}$ such that $ba = ac$. Comparings entries of both sides we
have $b_ia_i = a_ic_i$. By Proposition \ref{ilk}, $c_i\in
R_i^{qnil}$. Thus for each $i\in I$, $R_i$ is right qnil-duo.
Similarly, it is proven that for each $i\in I$, $R_i$ is left
qnil-duo.
\end{proof}

\noindent Recall that a ring $R$ is called {\it abelian} if every idempotent in $R$ is central.

\begin{thm}\label{abel} Let $R$ be a ring. Then the following hold.
\begin{enumerate}
\item[(1)] $ex - exe$ and $xe - exe\in R^{qnil}$ for every $x$ and  $e^2 = e\in R$.
\item[(2)] Right (resp., left) qnil-duo rings are abelian.
\item[(3)] Let R be a ring and $e\in$ Id$(R)$. If $R$ is right (resp., left) qnil-duo ring, then $eR$ and $(1 - e)R$ are right (resp., left) qnil-duo
rings. The converse holds if $e$ is central.
\end{enumerate}
\end{thm}
\begin{proof} (1) Let $t\in $ comm$(ex - exe)$. Then $t(ex -exe) = (ex -exe)t$.
So we have $(t(xe - exe))^2 = 0$. Hence $1 - (ex - exe)t$  is
invertible and so $ex - exe\in R^{qnil}$. Similarly, $xe - exe\in
R^{qnil}$.\\(2) Let $e^2 = e\in R$. By hypothesis, $
R^{qnil}e\subseteq e R^{qnil}$. By (1), $xe - exe\in R^{qnil}$ for
all $x\in R$. It implies for any $x\in R$, there exists $t\in
R^{qnil}$ such that $(xe - exe)e = et$. Multiplying the latter
equality by $e$ from the left we have $et = 0$. So $xe = exe$.
Similarly,  $ex = exe$  since $ex - exe\in R^{qnil}$ by (1). Hence
$R$ is abelian.\\ (3) It is clear by Theorem \ref{dik}.
\end{proof}
\begin{cor} Let $R$ be a right (resp., left) qnil-duo ring and $e\in$ Id$(R)$. Then the corner ring $eRe$ is a right (resp., left) qnil-duo ring.
\end{cor}
\begin{proof} The ring $R$ being right (resp., left) qnil-duo, $e$ is central in $R$ by Theorem \ref{abel}(2).
Hence Theorem \ref{abel}(3) completes the proof.
\end{proof}
\begin{thm} Every right (resp., left) qnil-duo ring is directly finite.
\end{thm}
\begin{proof} Let $R$ be a right qnil-duo ring and $a$, $b\in R$ with $ab = 1$. Set $e = 1 - ba$.
Then $e$ is an idempotent. By Theorem \ref{abel}, $e$ is central.
So $0 = ae = ea$. Hence $0 = a - ba^2$. Multiplying the latter by
$b$ from the right, we get $1 = ba$.
\end{proof}

\noindent There is a directly finite ring that is neither right nor left
qnil-duo.
\begin{ex} {\rm Consider the ring $R = M_2(\Bbb Z_2)$. Then $R$ is a directly finite ring but not abelian. Hence it is neither right nor left qnil-duo.}
\end{ex}
\noindent We apply Theorem \ref{abel} to show that full matrix rings and upper triangular matrix rings need not be right (resp., left)
qnil-duo. But there are some subrings of full matrix rings that are qnil-duo.
\begin{exs}\label{mat}{\rm \begin{enumerate} \item[(1)] For any ring $R$ and any positive integer $n$, $M_n(R)$ and $U_n(R)$ are neither right nor left qnil-duo.\
\item[(2)] If $R$ is commutative, then $V_n(R)$ is qnil-duo.
\item[(3)] $V_n(R[[x; \sigma]])$ is neither right nor left qnil-duo.\end{enumerate}}
\end{exs}
\begin{proof} (1) The rings $M_n(R)$ and  $U_n(R)$ are not abelian. By Theorem \ref{abel}(2), they are  neither right nor left qnil-duo.\\
(2) If $R$ is a commutative ring, $V_n(R)$ is also commutative,
therefore it is right and left qnil-duo.
\\(3) Let $R$ be a ring with an endomorphism $\sigma$. Assume that there exists $a_1\in R$ such that $\sigma(a_1)\notin a_1R$.
Let $A = \begin{bmatrix}x&x&x\\0&x&x\\0&0&x\end{bmatrix}\in
V_3(R[[x; \sigma]])^{qnil}$, $B =
\begin{bmatrix}a_1&a_2&a_3\\0&a_1&a_2\\0&0&a_1\end{bmatrix}\in
V_3(R[[x; \sigma]])$. Assume that there exists $D\in V_3(R[[x;
\sigma]])^{qnil}$  such that $AB = BD$. Then $(1, 1)$ entry of
$AB$ is $\sigma(a_1)x$ and that of $BD$ is $a_1xf(x)$ for some
$f(x)\in R[[x; \sigma]]$. This contradicts with the choice of
$\sigma$ and $a_1$. Therefore $V_3(R[[x; \sigma]])$ is not right
qnil-duo. Similarly, it can be shown that $V_3(R[[x; \sigma]])$ is
not left qnil-duo.
\end{proof}

\begin{thm}\label{d2q} Let $R$ be a local ring with $(R^{qnil})^2 = 0$. Then $R$ is right (resp., left) qnil-duo.
\end{thm}
\begin{proof} By Proposition \ref{nilö}(2), we have $R = U(R)\cup R^{qnil}$. We prove $R^{qnil}a\subseteq aR^{qnil}$.
Let $a\in R$, $b\in R^{qnil}$.
If $ba = 0$, then we are done since  $ba = 0 = a0\in aR^{qnil}$. Otherwise, i.e., if $ba\neq 0$, then we divide the proof in some cases.\\
Case I. Let  $a\notin R^{qnil}$. Then $a\in U(R)$. By Proposition \ref{nilö}(4), $a^{-1}ba\in R^{qnil}$ since $b\in R^{qnil}$. Then $ba = a(a^{-1}ba)\in aR^{qnil}$.\\
Case II. Let  $a\in R^{qnil}$. By hypothesis, $ba = 0$, this contradicts with $ba\neq 0$. \\
Therefore $R$ is right qnil-duo ring. Similarly, we may prove $aR^{qnil}\subseteq R^{qnil}a$ for each $a\in R$.
\end{proof}

\noindent As an illustration of Theorem \ref{d2q}, we give the following
examples. Also, the condition $(R^{qnil})^2 = 0$ in Theorem
\ref{d2q} is not superfluous.

\begin{ex}{\rm (1) Consider the ring $R = \left\{\begin{bmatrix}a&b&c\\0&a&0\\0&0&a\end{bmatrix}\in D_3(\Bbb Z_4)\right\}$. Then \newline $R^{qnil} = \left\{\begin{bmatrix}a&b&c\\0&a&0\\0&0&a\end{bmatrix}\in R\mid a\in 2\Bbb Z_4, b, c\in \Bbb Z_4\right\}$. So $(R^{qnil})^2 = 0$. By Theorem \ref{d2q}, $R$ is qnil-duo. \\
(2) Let $R$ denote the ring in Examples \ref{ör1}. Then $R^{qnil} = \{0, 2, x, y, 2+x, 2+y, 2 + x + y, x + y\}$ and $(R^{qnil})^2\neq 0$ and $2 + x$ belongs to $R^{qnil}$ since it is nilpotent and $(2 + x)^2\neq 0$. Since $R$ is local and $R^{qnil}$ does not contain invertible elements, $R^{qnil} = J(R)$. To complete the proof we may assume that $x$, $y\in R^{qnil}$. Then $xy = 2$ and $xy\in R^{qnil}y$. It is easily checked that there is no $t\in R^{qnil}$ such that $xy = yt\in yR^{qnil}$. Hence $R$ is not right qnil-duo. Compare to Theorem \ref{d2q}.
}
\end{ex}

\noindent Note that by Theorem \ref{d2q}, if $R$ is a division ring, $D_2(R)$ is a qnil-duo ring.  One may ask whether $D_2(R)$ is qnil-duo over a domain $R$. The following example answers negatively.
\begin{ex}\label{we}{\rm  Consider the ring $D_2(R[[x]])$ in \cite[Example 1.4(1)]{KY}. It is proved that $D_2(R[[x]])$ is neither right nor left normal  property of elements on Jacobson radical. Since $J(D_2(R[[x]])) = D_2(R[[x]])^{qnil}$, $D_2(R[[x]])$ is neither right nor left qnil-duo.}
\end{ex}
\begin{thm}\label{ilkt} Let $R$ be a domain. If $D_2(R)$ is right (resp., left) qnil-duo, then $R$ is right (resp., left) qnil-duo.
\end{thm}
\begin{proof} Assume that $D_2(R)$ is right qnil-duo. Let $a\in R$ and $b\in R^{qnil}$. Consider
$A = \begin{bmatrix}a&0\\0&a\end{bmatrix} \neq 0$, $B =
\begin{bmatrix}b&0\\0&b\end{bmatrix} \neq 0$. Let $X =
\begin{bmatrix}x&y\\0&x\end{bmatrix}\in $ comm$(B)$. Then $I_2 - BX$
is invertible since $1 - bx$ is invertible in $R$. Hence $BA\in
D_2(R)^{qnil}A$. There exists $C =
\begin{bmatrix}c&d\\0&c\end{bmatrix}\in  D_2(R)^{qnil}$ such that
$BA = AC$. Then $ba = ac$ and $ad = 0$. By hypothesis $d = 0$. By
Lemma \ref{nil}(3), $c\in R^{qnil}$. It follows that $ba = ac\in
aR^{qnil}$. Hence $R^{qnil}a\subseteq aR^{qnil}$.
\end{proof}
\noindent Recall that a ring $R$ is said to have {\it stable range 1} if for
any $a$, $b\in R$ satisfying $aR + bR = R$, there exists $y\in R$
such that $a + by$ is right invertible (cf. \cite{Va}). In
\cite{Ni}, a ring $R$ is called {\it exchange} if for any $x\in
R$,  there exists $e\in$ Id$(R)$ such that $e\in Rx$ and $1 - e\in
R(1 - x)$, and it is proved that for an abelian ring $R$, $R$ is
exchange if and only if it is clean, and $R$ is exchange if and
only if idempotents lift modulo every left (or right) ideal.
\begin{thm} The following hold. \begin{enumerate} \item[(1)] Right (resp., left) qnil-duo exchange rings have stable range 1.
\item[(2)] Right (resp., left) qnil-duo regular rings (in the sense of von Neumann) are strongly regular.
\end{enumerate}
\end{thm}
\begin{proof} (1) Let $R$ be a right qnil-duo exchange ring. By Theorem \ref{abel}, $R$ is abelian. Hence \cite[Theorem 6]{Yu}
implies $R$ has stable range 1.\\(2) Let $R$ be a qnil-duo regular ring and $a\in R$. There exists $b\in R$ such that $a = aba$. Then $ab = (ab)^2$, $ba = (ba)^2\in$ Id$(R)$.  By Theorem \ref{abel}, $ab$ is central. So $a = aba = a^2b$. Hence $R$ is strongly regular.
\end{proof}

\noindent Let $R$ be a ring. The Jacobson radical of the
polynomial ring $R[x]$ is $J(R[x]) = N[x]$ where $N = J(R[x])\cap
R$ is a nil ideal of $R$. Then $N\subseteq R^{qnil}$ and
$J(R[[x]]) = xR[[x]]$. Therefore, ${R[[x]]}^{qnil} = xR[[x]]$. One
may wonder whether or not $R[x]$ and $R[[x]]$ are qnil-duo. The
following example shows that $R[x]$ and $R[[x]]$ need not be right
qnil-duo.

\begin{ex}\label{weyson}{\rm (1) Let $F$ be a field, $R = M_n(F)$ and consider the ring $R[x]$.
Observe that $M_n(F[x])$ is not right (or left) qnil-duo for any
positive integer $n\geq 2$ by Examples \ref{mat}. It follows that
$R[x]$ is not right (or left) qnil-duo since $M_n(F)[x] \cong
M_n(F[x])$.\\
\noindent (2) Let $R = A/(ab-ba-1)$ denote the Weyl algebra discussed
in \cite[Example 1.2(2)]{KY}. Let $S = R[[x]]$. Then $S^{qnil} =
xR[[x]] = J(R[[x]])$, $R$ is a domain and $R[[x]]$ is abelian. It
is proved that $R[[x]]$ is neither right normal nor left normal on
$J(R)$. Therefore, $R[[x]]$ is neither right qnil-duo nor left
qnil-duo.}
\end{ex}
\begin{thm} Let $R$ be  an algebra over a commutative ring $S$. Consider the {\it Dorroh extension} (or {\it ideal extension}) $I(R, S)$ of $R$ by $S$. If $I(R, S)$ is right qnil-duo, then so is $R$.
\end{thm}
\begin{proof}  Assume that $I(R, S)$ is right qnil-duo. Let $a\in R$, $b\in R^{qnil}$. Then $(a, 0)\in I(R, S)$ and $(b, 0)\in I(R, S)^{qnil}$.
Indeed, let $(x, y)\in $ comm$(b, 0)$. By Lemma \ref{IR}, $x\in $
comm$(b)$. Since $R$ is an algebra over $S$, we have $x+y\in $
comm$(b)$. Then $1 + b(x+y)$ is invertible in $R$ with inverse
$t$. Again by Lemma \ref{IR}, $(0, 1) + (b, 0)(x, y)$ is
invertible in $I(R, S)$ with the inverse $(t-1, 1)$. Then $(b,
0)(a, 0)\in I(R, S)^{qnil}(a, 0)$. There exists $(c, s)\in I(R,
S)^{qnil}$ such that $(b, 0)(a, 0) = (a, 0)(c, s)$. So $ba = a(c +
s)$. To complete the proof we show $c + s\in R^{qnil}$. Let $x\in
$ comm$(c+s)$. Then $cx+sx=sc+xs$. Since $R$ is an algebra over
$S$, $sx=xs$, this implies $cx=xc$, and so $x\in $ comm$(c)$.
Hence $(x,0)\in $ comm$(c,s)$. Since $(c, s)\in I(R, S)^{qnil}$,
$(0,1)+(c,s)(x,0)$ is invertible in $I(R,S)$. Thus $1+(c+s)x$ is
invertible in $R$ by Lemma \ref{IR}(2). So $c + s\in R^{qnil}$.
Therefore $R$ is right qnil-duo.
\end{proof}
\begin{prop}Let $R$ be a ring and $S$ a subring of $R$. If $T[R, S]$ is right qnil-duo, then so are $R$ and $S$. The converse holds if $S^{qnil}\subseteq
R^{qnil}$.
\end{prop}
\begin{proof} Assume that $T[R, S]$ is right qnil-duo. Let $a\in R$, $b\in R^{qnil}$.
Let $A = (a, 0, 0, 0, \dots)$, $B = (b, 0, 0, 0, \dots)\in
T[R,S]$. By Proposition \ref{TRS}, $B\in T[R,S]^{qnil}$. By
supposition there exists $C = (c_1,c_2,\dotsb,c_m,t,t,\dots.)\in
T[R,S]^{qnil}$ such that $BA = AC$. Hence $ba = ac_1$. By
Proposition \ref{TRS}, $c_1\in R^{qnil}$. Similarly, let $s\in S$,
$t\in S^{qnil}$ and $C =(0, s, s, s, s, \dots,s ,\dots)$, $D = (0,
t, t, t, \dots.)\in T[R,S]$. By Proposition \ref{TRS}, $D\in
T[R,S]^{qnil}$. There exists $D' = (d_1, d_2, d_3, \dots., d_l, u,
u, u, \dots)\in T[R, S]^{qnil}$ such that $DC = CD'$. By
Proposition \ref{TRS}, $u\in S^{qnil}$ and $ts = su\in sS^{qnil}$.

Suppose that $R$ and $S$ are right qnil-duo and $S^{qnil}\subseteq
R^{qnil}$. Let $A\in T[R, S]$, $B\in T[R, S]^{qnil}$ where $A = (
a_1, a_2,\dotsb, a_n, s, s,\dotsb)$, $B = (b_1,b_2,\dots,b_m, t,
t, \dots)$, we prove $BA = AC$ for some $C\in T[R, S]^{qnil}$. By
Proposition \ref{TRS}, $b_i\in R^{qnil}$ for $i = 1, 2, \dots, m$
and $t\in S^{qnil}$.
By supposition $b_i\in R^{qnil}$ implies $b_ia_i = a_ic_i$ for some $c_i\in R^{qnil}$. We divide the proof in some cases:\\
Case I. $n\leq m$. Then $b_ia_i\in R^{qnil}a_i$. Since $R$ is
right qnil-duo, there exist $c_i\in R^{qnil}$ such that $b_ia_i =
a_ic_i$ for each $1\leq i\leq n$. For $n+1\leq i\leq m$, $b_is\in
R^{qnil}s$. There exist $c_i\in R^{qnil}$ such that $b_is = sc_i$.
For  $ts\in S^{qnil}s$, there exists $l\in S^{qnil}$ such that $ts
= sl\in sS^{qnil}$. Let $C = (c_1, c_2, c_3, \dots, c_m, l, l, l,
\dots)$.
By Proposition \ref{TRS}(2), $C\in T[R,S]^{qnil}$. Then $BA = AC\in A T[R,S]^{qnil}$.  \\
Case II. $n > m$. Let $1\leq i\leq m$. Then $b_ia_i\in R^{qnil}a_i$ and since $R$ is right qnil-duo, there exist $c_i\in R^{qnil}$ such that $b_ia_i = a_ic_i$.
For $m+1\leq i\leq n$, $ta_i\in S^{qnil}a_i$.
By $S^{qnil}\subseteq R^{qnil}$,  we have $ta_i = a_ic_i \in a_iR^{qnil}$ for some $c_i\in R^{qnil}$. For $ts\in S^{qnil}s$, by supposition there exists $l\in S^{qnil}$
such that $ts = sl\in sS^{qnil}$. Let $C = (c_1, c_2, c_3, \dots, c_n, l, l, l, \dots)$. By Proposition \ref{TRS}(2), $C\in T[R, S]^{qnil}$. Then $BA = AC$. Hence $T[R,S]^{qnil}A\subseteq AT[R,S]^{qnil}$. It completes the proof.
\end{proof}
\begin{thm} \begin{enumerate}
\item[(1)] Let $H(R; \alpha)$ be a Skew Hurwitz series ring over a
ring $R$. If $H(R; \alpha)$ is right qnil-duo, then $R$ is right
qnil-duo. \item[(2)] Let $R[[x; \alpha]]$ be a skew formal power
series ring over a ring $R$. If $R[[x; \alpha]]$ is right
qnil-duo,  then $R$ is right qnil-duo.
\end{enumerate}
\end{thm}
\begin{proof} Suppose that $H(R; \alpha)$ is a right
qnil-duo ring. Let $a\in R^{qnil}$ and $b\in R$. By the definition
of $\epsilon$ and Proposition \ref{HR}, there exist $f(x)$,
$g(x)\in H(R; \alpha)$ with $f(x)\in H(R; \alpha)^{qnil}$ and
$\epsilon(f(x)) = a$, $\epsilon(g(x)) = b$. There exists $h(x) =
c_0 + c_1x + c_2x^2 + \dots\in H(R; \alpha)^{qnil}$ such that
$f(x)g(x) = g(x)h(x)$. Hence $\epsilon(f(x)g(x)) =
\epsilon(g(x)h(x))$ implies
$ab = bc_0\in bR^{qnil}$. Thus $R^{qnil}b \subseteq bR^{qnil}$.\\
(2) Similar to that of (1).
\end{proof}
\section{Some subrings of matrix rings}
\noindent Besides, for any ring $R$ and any positive integer
$n\geq 2$, $M_n(R)$ is not right (or left) qnil-duo, in this
section, quasinilpotent elements of some subrings of full matrix
rings are determined for the purpose of the use whether or not
their subrings to be right (or left) qnil-duo.

{\bf The rings $L_{(s,t)}(R)$:} Let $R$ be a ring and $s$, $t\in
C(R)$.\\ Let $L_{(s,t)}(R) = \left
\{\begin{bmatrix}a&0&0\\sc&d&te\\0&0&f\end{bmatrix}\in M_3(R)\mid
a, c, d, e, f\in R\right \}$, where the operations are defined as
those in $M_3(R)$. Then $L_{(s,t)}(R)$ is a subring of $M_3(R)$.
\begin{lem}\label{inv1} Let $A = \begin{bmatrix}a&0&0\\sc&d&te\\0&0&f\end{bmatrix}\in L_{(s,t)}(R)$. Then the following hold.
\begin{enumerate}\item[(1)] $A$ is invertible in  $L_{(s,t)}(R)$ if and only if $a$, $d$ and $f$ are invertible in $R$.
\item[(2)] If  $a$, $d$, $f\in R^{qnil}$, then $A\in L_{(s,t)}(R)^{qnil}$.
\end{enumerate}
\end{lem}
\begin{proof} (1) One way is clear. Let $A = \begin{bmatrix}a&0&0\\sc&d&te\\0&0&f\end{bmatrix}\in L_{(s, t)}(R)$.
Assume that $a$, $d$ and $f$ are invertible with $ax = xa = 1$,
$dz = zd =1$ and $fv = vf = 1$ where $x, z, v\in R$. Consider
$B = \begin{bmatrix}x&0&0\\sy&z&tu\\0&0&v\end{bmatrix}\in L_{(s, t)}(R)$ where $y = -zcx$ and $u = -zev$. Then $AB = BA = I_3$.\\
(2)  Assume that $a$, $d$, $f\in R^{qnil}$. We prove that $A\in
L_{(s, t)}(R)^{qnil}$. Let $B =
\begin{bmatrix}x&0&0\\sy&z&tu\\0&0&v\end{bmatrix}\in L_{(s,
t)}(R)$ with $B\in $ comm$(A)$. It is easily checked that $x\in $
comm$(a)$, $z\in $ comm$(d)$, $v\in $ comm$(f)$. Then $1 + ax$, $1
+ dz$, $1 + fv$ are invertible in $R$. By (1), $I_3 + AB =
\begin{bmatrix}1 + ax&0&0\\scx + sdy&1 + dz&tdu + tev\\0&0&1 +
fv\end{bmatrix}$ is invertible. So $A\in L_{(s, t)}(R)^{qnil}$.
\end{proof}
\begin{lem}\label{final} Let $A = \begin{bmatrix}a&0&0\\sc&d&te\\0&0&f\end{bmatrix}\in L_{(s,t)}(R)$. Then the following hold.
\begin{enumerate}
    \item[(1)] If $A\in L_{(0,t)}(R)^{qnil}$, then  $a\in R^{qnil}$.
    \item[(2)] If $A\in L_{(s,0)}(R)^{qnil}$, then  $f\in R^{qnil}$.
    \item[(3)] $A\in L_{(0,0)}(R)^{qnil}$ if and only if  $a,d,f\in R^{qnil}$.
\end{enumerate}
\end{lem}
\begin{proof} (1) Let
$A=\begin{bmatrix}a&0&0\\0&d&te\\0&0&f\end{bmatrix}\in
L_{(0,t)}(R)^{qnil}$ and $x\in $ comm$(a)$. Consider
$B=\begin{bmatrix}x&0&0\\0&0&0\\0&0&0\end{bmatrix}\in
L_{(0,t)}(R)$. Then $B\in $ comm$(A)$. Since $A\in
L_{(0,t)}(R)^{qnil}$, $I+AB$ is invertible in $L_{(0,t)}(R)$. By
Lemma \ref{inv1}(1), $1+ax\in U(R)$. Therefore $a\in R^{qnil}$.\\
(2) Similar to the proof of (1).\\
(3) The sufficiency follows from Lemma \ref{inv1}(2). For the
necessity, $a,f\in R^{qnil}$ by (1) and (2), respectively. Also,
by the similar discussion in (1), we obtain $d\in R^{qnil}$.
\end{proof}
\begin{thm}\label{guz} Let $R$ be a ring. If $L_{(0, t)}(R)$ is right qnil-duo, then $R$ is a right qnil-duo ring.
\end{thm}
\begin{proof} Assume that $L_{(0, t)}(R)$ is right qnil-duo and let $a\in R$ and $ b\in R^{qnil}$. Consider $A =
\begin{bmatrix}a&0&0\\0&0&0\\0&0&0\end{bmatrix}$, $B = \begin{bmatrix}b&0&0\\0&0&0\\0&0&0\end{bmatrix}\in L_{(0,t)}(R)$.
By Lemma \ref{inv1}, $B\in  L_{(0,t)}(R)^{qnil}$. By supposition there exists $B' = \begin{bmatrix}x&0&0\\0&z&tu\\0&0&v\end{bmatrix}\in
L_{(0,t)}(R)^{qnil}$ such that $BA = AB'$. It implies $ba = ax$. By Lemma \ref{final}(1), $x\in R^{qnil}$. Hence $ba = ax\in
aR^{qnil}$. Thus $R^{qnil}a\subseteq aR^{qnil}$.
\end{proof}
\noindent There are  right qnil-duo  rings $R$ such that the rings  $L_{(s, t)}(R)$ need not be right qnil-duo as shown below.
\begin{ex}{\rm The ring $L_{(1,1)}(\Bbb Z_4)$ is not  right qnil-duo.
}\end{ex}
\begin{proof} Let $A = \begin{bmatrix}0&0&0\\1&2&1\\0&0&3\end{bmatrix}\in L_{(1, 1)}(\Bbb Z_4)$ and $B = \begin{bmatrix}2&0&0\\1&2&3\\0&0&2\end{bmatrix}\in L_{(1, 1)}(\Bbb Z_4)^{qnil}$. Assume that there exists $C =
\begin{bmatrix}x&0&0\\y&z&v\\0&0&u\end{bmatrix}\in L_{(1, 1)}(\Bbb Z_4)^{qnil}$ such that $BA = AC$. Then $BA =
\begin{bmatrix}0&0&0\\2&0&3\\0&0&2\end{bmatrix}$ and $AC = \begin{bmatrix}0&0&0\\x+2y&2z&2v+u\\0&0&3u\end{bmatrix}$. $BA =
AC$ implies $ 3 = 2v +u$ and $2=3u$. These equations lead us a contradiction. Hence $L_{(1,1)}(\Bbb Z_4)$ is not right qnil-duo.
\end{proof}
%However, there are right qnil-duo subrings  $L_{(s, t)}(D_3(R))$
%of  $L_{(s, t)}(R)$. Consider the subring $L_{(s, t)}(D_3(R)) =
%\left \{\begin{bmatrix}a&0&0\\nsa&a&nta\\0&0&a\end{bmatrix}\in
%M_3(R)\mid a\in R, n\in \Bbb Z\right \}$ of $L_{(s, t)}(R)$.
%\begin{thm}\label{guzel} Let $R$ be a ring. Then $R$ is a right qnil-duo ring if and only if so is  $L_{(s, t)}(D_3(R))$.
%\end{thm}
%\begin{proof} Assume that $R$ is a right qnil-duo ring and let $A = \begin{bmatrix}a&0&0\\nsa&a&nta\\0&0&a\end{bmatrix}\in L_{(s, t)}(D_3(R))$ and $B =
%\begin{bmatrix}x&0&0\\msx&x&mtx\\0&0&x\end{bmatrix}\in L_{(s, t)}(D_3(R))^{qnil}$. By Lemma \ref{inv1}, there exists $x'\in R^{qnil}$ such that $xa = ax'$. Again, by Lemma %\ref{inv1}, $B' =
%\begin{bmatrix}x'&0&0\\msx'&x'&mtx'\\0&0&x'\end{bmatrix}\in L_{(s, t)}(D_3(R))^{qnil}$. It is easily checked that $BA = AB'$ .  Converse is clear by the proof of Theorem %\ref{guz}.
%\end{proof}
{\bf The rings $H_{(s,t)}(R)$:} Let $R$ be a ring and  $s$, $t\in C(R)$ be invertible in $R$. Let\begin{center} $H_{(s,t)}(R) = \left \{\begin{bmatrix}a&0&0\\c&d&e\\0&0&f
\end{bmatrix}\in M_3(R)\mid a, c, d, e, f\in R, a - d = sc, d - f = te\right \}$.\end{center} Then $H_{(s,t)}(R)$ is a subring of $M_3(R)$.
\begin{lem}\label{son1} Let  $A = \begin{bmatrix}a&0&0\\c&d&e\\0&0&f\end{bmatrix}$, $B = \begin{bmatrix}x&0&0\\y&z&u\\0&0&v\end{bmatrix}\in H_{(s,t)}(R)$. Then\begin{enumerate}\item[(1)] $AB = BA$ if and only if $ax = xa$, $dz = zd$, $fv = vf$.\item[(2)] $A$ is invertible with inverse $B$ if and only if $ax = xa = 1$, $dz = zd = 1$, $fv = vf = 1$.\item[(3)] $A\in H_{(s,t)}(R)^{qnil}$ if and only if $a$, $d$, $f\in R^{qnil}$.
\end{enumerate}
\end{lem}
\begin{proof} (1) The necessity is clear. For the sufficiency, suppose that $ax = xa$, $dz = zd$, $fv = vf$. The matrix $AB$ has $cx + dy$ as $(2, 1)$ entry, $du + ev$ as $(2, 3)$ entry and $BA$ has $ya + zc$ as $(2,1)$ entry, $ze + uf$ as $(2, 3)$ entry. To show $AB = BA$ it is enough to get $cx + dy = ya + zc$ and $du + ev = ze + uf$. Now $scx + sdy = ax + d(sy - x) = ax - dz = xa - za + za - dz = sya + szc$. So  $cx + dy = ya + zc$ since $s$ is invertible. Similarly, we get $du + ev = ze + uf$.\\
(2) One way is clear. Assume that $ax = xa = 1$, $dz = zd = 1$, $fv = vf = 1$. Let $B\in H_{(s,t)}(R)$ with $y = -zcx$ and $u = -zev$. Then $AB = BA = I_3$.\\
(3) Assume that $A\in H_{(s,t)}(R)^{qnil}$. Let $x\in$ comm$(a)$, $y\in$ comm$(f)$. Let $D = \begin{bmatrix}x&0&0\\s^{-1}x&0&-t^{-1}y\\0&0&y\end{bmatrix}$. Then $D\in$ comm$(A)$. In fact, $scx = (a - d)x$ and $tey = (d - f)y$. Hence $I_3 + AD$ is invertible in $H_{(s, t)}(R)$. It follows that $1 + ax$, $1 + fy\in U(R)$. So $a, f\in R^{qnil}$. As for $d\in R^{qnil}$, let $r\in$ comm$(d)$ and $D = \begin{bmatrix}0&0&0\\-s^{-1}r&r&t^{-1}r\\0&0&0\end{bmatrix}$. Then $D\in$ comm$(A)$. By assumption $I_3 + AD\in U(H_{(s, t)}(R))$. Hence $1 + dr\in U(R)$. Hence $d\in R^{qnil}$. Conversely, suppose that $a$, $d$, $f\in R^{qnil}$. Let $B\in$ comm$(A)$. Then $x\in$ comm$(a)$, $z\in$ comm$(d)$ and $v\in$ comm$(f)$. By supposition, $1 + ax$, $1 + dy$ and $1 + fv$ are invertible. By part (2), $I_3 + AB\in U(H_{(s, t)}(R))$. Hence $A\in H_{(s, t)}(R)^{qnil}$. This completes the proof.
\end{proof}
\begin{thm} Let $R$ be a ring. Then $R$ is right qnil-duo if and only if $H_{(s, t)}(R)$ is right qnil-duo.
\end{thm}
\begin{proof} Assume that $R$ is a right qnil-duo ring.
Let $A = \begin{bmatrix}a&0&0\\c&d&e\\0&0&f\end{bmatrix}\in H_{(s,
t)}(R)$ and $B =
\begin{bmatrix}x&0&0\\y&z&u\\0&0&v\end{bmatrix}\in
H_{(s,t)}(R)^{qnil}$. By Lemma \ref{son1}, $x$, $z$, $v\in
R^{qnil}$. There exist $x'$, $z'$, $v'\in R^{qnil}$ such that $xa
= ax'$, $zd = dz'$, $vf = fv'$. Let $y' = s^{-1}(x' - z')$ and $u'
= t^{-1}(z' - v')$ and $B' =
\begin{bmatrix}x'&0&0\\y'&z'&u'\\0&0&v'\end{bmatrix}$. Then $B'\in
H_{(s, t)}(R)^{qnil}$. We next show that $BA = AB'$. It is enough
to see $ya + zc = cx' + dy'$ and $ze + uf = du' + ev'$. We start
with, $cx' + dy' = cx' + ds^{-1}x' - ds^{-1}z'$. Multiplying the
latter from the left by $s$ and using the fact that $s$ is
central, we have $s(cx' + dy') = scx' + dx' - dz' = (sc + d)x' -
zd = ax' - zd = xa - zd = (xa - za) + (za - zd) = sya + szc = s(ya
+ zc)$. Since $s$ is invertible, $ya + zc = cx' + dy'$. Similarly,
$du' + ev' = dt^{-1}z' -  dt^{-1}v' + ev'$. Multiplying the latter
from the left by $t$ and using the fact that $t$ is central, we
have $t(du' + ev') = dz' - dv' + tev' = zd + (te - d)v' = zd - fv'
= zd - vf = zd - zf + zf - vf = z(d - f) + (z - v)f = t(ze + uf)$.
By using invertibility of $t$, we get $du' + ev' = ze + uf$.
Conversely, suppose that $H_{(s, t)}(R)$ is a right qnil-duo ring.
Let $a\in R$ and $ b\in R^{qnil}$. Consider $A = aI_3$, $B = bI_3\in H_{(s,t)}(R)$.
By Lemma \ref{son1}, $B\in  H_{(s,t)}(R)^{qnil}$. By supposition
there exists $B' =
\begin{bmatrix}x&0&0\\y&z&u\\0&0&v\end{bmatrix}\in
H_{(s,t)}(R)^{qnil}$ such that $BA = AB'$. It implies $ba = ax$.
Again by Lemma \ref{son1}, $x\in R^{qnil}$. Hence $ba = ax\in
aR^{qnil}$. Thus $R^{qnil}a\subseteq aR^{qnil}$.
\end{proof}

\bigskip

{\bf Generalized matrix rings:} Let $R$ be a ring and $s\in U(R)$. Then  $\begin{bmatrix}
R&R\\R&R\end{bmatrix}$ becomes a ring denoted by $K_s(R)$ with
addition defined componentwise and multiplication defined in
\cite{Kr} by
$$\begin{bmatrix} a_1&x_1\\y_1&b_1\end{bmatrix}\begin{bmatrix}
a_2&x_2\\y_2&b_2\end{bmatrix} = \begin{bmatrix}a_1
a_2+sx_1y_2&a_1x_2+x_1b_2\\y_1a_2+b_1y_2&sy_1x_2+b_1b_2\end{bmatrix}.$$
In \cite{Kr}, $K_s(R)$ is called a {\it generalized matrix ring
over $R$}.
\begin{lem}\label{inv} Let $R$ be a ring. Then the following hold.
\begin{enumerate}
\item[(1)] \begin{center} $U(K_0(R)) = \left \{\begin{bmatrix}a&b\\c&d\end{bmatrix}\in K_0(R)\mid a, d\in U(R)\right\}$.\end{center}\item[(2)]\begin{center} $C(K_0(R))= \left \{
\begin{bmatrix}a&0\\0&a\end{bmatrix}\in K_0(R)\mid a\in C(R)\right\}$.
\end{center}
\end{enumerate}
\end{lem}
\begin{proof} (1) Let $A  =  \begin{bmatrix}a&b\\c&d\end{bmatrix}\in U(K_0(R))$.
There exists $B = \begin{bmatrix}x&y\\z&t\end{bmatrix}\in K_0(R)$
such that $AB = BA = I$, where $I$ is the identity matrix. Then we
have $ax = xa = 1$ and $dt = td = 1.$ So $a$ and $d$ are invertible. Conversely, let $A  =  \begin{bmatrix}a&b\\c&d\end{bmatrix}\in K_0(R)$ with $a$, $d\in U(R)$. Let $x = a^{-1}$, $t = d^{-1}$, $k = -a^{-1}bd^{-1}$ and $l = -d^{-1}ca^{-1}$. Then $B =
\begin{bmatrix}x&k\\l&t\end{bmatrix}$ is the inverse of $A$ in $K_0(F)$.\\(2) Let $A =
\begin{bmatrix}a&b\\c&d\end{bmatrix}\in C(K_0(R))$. By commuting
$A$ in turn with the matrices
$\begin{bmatrix}1&0\\0&0\end{bmatrix}$ and
$\begin{bmatrix}0&1\\0&0\end{bmatrix}$
in $K_0(R)$ we reach at $A
= \begin{bmatrix}a&0\\0&a\end{bmatrix}$. For the converse, let
$A = \begin{bmatrix}a&0\\0&a\end{bmatrix}\in K_0(R)$ where $a\in C(R)$. Then clearly, $A$ commutes with every element
of $K_0(R)$. So $A\in C(K_0(R))$.
\end{proof}
\begin{prop}\label{qnil} Let $R$ be a ring and $A = \begin{bmatrix}a&b\\c&d\end{bmatrix}\in K_0(R)$. If $a$, $d\in R^{qnil}$, then $A\in K_0(R)^{qnil}$.
\end{prop}
\begin{proof} Suppose that $a$, $d\in R^{qnil}$. Let $B = \begin{bmatrix}x&y\\z&t\end{bmatrix}\in K_0(R)$ with $B\in $ comm$(A)$.
Then $x\in $ comm$(a)$, $t\in $ comm$(d)$. Let $r = 1 + ax$, $v =
1 + dt$, $s = ay + bt$ and $u = cx + dz$. By assumption, $r = 1 +
ax$ and $v = 1 + dt$ are invertible in $R$. Let  $k =
-r^{-1}sv^{-1}$ and $l = -v^{-1}ur^{-1}$. Then $I_2 + AB =
\begin{bmatrix}r&s\\u&v\end{bmatrix}$ is invertible with the inverse
$C = \begin{bmatrix}r^{-1}&k\\l&v^{-1}\end{bmatrix}$.
\end{proof}
\noindent Let $R$ be a ring, $a$, $b\in R$. Define $l_a-r_b \colon
R\rightarrow R$ by $(l_a-r_b)(r) = ar - rb$ and $l_b-r_a \colon
R\rightarrow R$ by $(l_b-r_a)(r) = br - ra$. In \cite{BDD}, a
local ring $R$ is called {\it bleached} if for any $a\in J(R)$ and
any $b\in U(R)$, the abelian group endomorphisms $l_b - r_a$ and
$l_a - r_b$ of $R$ are surjective. Such rings is called {\it
uniquely bleached} if the appropriate maps are injective as well
as surjective.  In \cite{SCB}, $R$ is {\it a weakly bleached ring}
provided that for any $a\in J(R)$, $b\in 1 + J(R)$, $l_a - l_b$
and $l_b - l_a$  are surjective and it is proved that matrices
over 2-projective free rings are strongly J-clean. It is proved
that all upper triangular matrices over bleached local rings are
strongly clean. In \cite[Example 2]{Ni} and \cite[Theorem 18]{BDD}
it is proved that a local ring $R$ is weakly bleached if and only
if the $2\times 2$ upper triangular matrix ring $U_2(R)$ is
strongly clean. In the preceding, maps of the form $l_a - r_b$
play a central role. In this vein, we make use of the abelian
group endomorphisms $l_a - r_b$ to get the following result as
partly the converse of Proposition \ref{qnil}.

\begin{thm} Let $R$ be a ring and $A = \begin{bmatrix}a&b\\c&d\end{bmatrix}\in K_0(R)^{qnil}$. If for any  $x\in $ comm$(a)$ and $y\in $ comm$(d)$ and for the abelian group endomorphisms $l_y - r_x$ and $l_x - r_y$, $b\in$ Ker$(l_x - r_y)$ and $c\in$ Ker$(l_y - r_x)$, then $a$, $d\in R^{qnil}$.
\end{thm}
\begin{proof} Assume that $A = \begin{bmatrix}a&b\\c&d\end{bmatrix}\in K_0(R)^{qnil}$,
$x\in $ comm$(a)$ and $y\in $ comm$(d)$, for $l_y - r_x$ and $l_x
- r_y$, $b\in$ Ker$(l_x - r_y)$ and $c\in$ Ker$(l_y - r_x)$. Then
$b\in$ Ker$(l_x - r_y)$ implies $(l_x - r_y)(b) = 0$. So $xb =
by$. $c\in$ Ker$(l_y - r_x)$ implies $(l_y - r_x)(c) = 0$. So $yc
= cx$. Let $B = \begin{bmatrix}x&0\\0&y\end{bmatrix}\in K_0(R)$.
Then $xb = by$ and $yc = xc$ give rise to $B\in $ comm$(A)$. By
hypothesis, $I_2 + AB$ is invertible. Then Lemma \ref{inv} implies
$1 - ax$ and $1 - dy$ are invertible. Hence $a$, $d\in R^{qnil}$.
\end{proof}
\noindent We may determine the set $K_0(R)^{qnil}$ for some rings $R$.
\begin{prop}\label{nnil}\begin{enumerate} \item[(1)] If $R$ is a local ring, then $A = \begin{bmatrix}a&b\\c&d\end{bmatrix}\in K_0(R)^{qnil}$ if and only if $a$, $d\in R^{qnil}$.
\item[(2)] Let $R$ be a ring. Then $A = \begin{bmatrix}a&0\\0&d\end{bmatrix}\in K_0(R)^{qnil}$ if and only if $a$, $d\in R^{qnil}$.
\end{enumerate}
\end{prop}
\begin{proof} (1) Assume that $R$ is a local ring and $A = \begin{bmatrix}a&b\\c&d\end{bmatrix}\in K_0(R)^{qnil}$, and $d\notin R^{qnil}$.
By Proposition \ref{nilö}, $d\in U(R)$. In this case, $1 + d$ can not belong to $U(R)$. By Lemma \ref{inv}, $I + A$ can not belong to $U(K_0(R))$.
This contradicts $A\in K_0(R)^{qnil}$. It follows that $d\in R^{qnil}$. Similarly, we obtain $a\in R^{qnil}$.  The converse is clear by Proposition \ref{qnil}.\\
(2) Clear.
\end{proof}

\noindent There are some classes of rings $R$ in which $K_0(R)$ being a
right qnil-duo ring implies $R$ being a right qnil-duo ring.
\begin{thm}  Let $R$ be a ring. Then $K_0(R)$ being a right qnil-duo ring implies $R$ being a right qnil-duo ring if $R$ is one of the following rings.
\begin{enumerate}\item[(1)] $R$ is local.\item[(2)] $R$ has no nonzero zero divisors.
\end{enumerate}
\end{thm}
\begin{proof} (1) Let $R$ be a local ring. Assume that $K_0(R)$ is a right qnil-duo ring. Let $a\in R$, $b\in R^{qnil}$. Consider $A = \begin{bmatrix}a&0\\0&a\end{bmatrix}$, $X = \begin{bmatrix}b&0\\0&b\end{bmatrix}\in K_0(R)$. By Proposition \ref{qnil}, $X\in K_0(R)^{qnil}$. There exists $X' = \begin{bmatrix}x'&y'\\z'&t'\end{bmatrix}\in K_0(R)^{qnil}$ such that $XA = AX'$. Hence $ba = ax'$. By Proposition \ref{nnil}, $x'\in R^{qnil}$. So $ba = ax'\in aR^{qnil}$.\\
(2) Let $R$ be a ring having no nonzero zero divisors. Assume that
$K_0(R)$ is a right qnil-duo ring. Let $a\in R$, $b\in R^{qnil}$.
If $a = 0$ or $b = 0$, there is nothing to do. Let $a\neq 0$ and
$b\neq 0$ and consider $A = \begin{bmatrix}a&0\\0&a\end{bmatrix}$,
$B = \begin{bmatrix}b&0\\0&b\end{bmatrix}\in K_0(R)$. By
Proposition \ref{qnil}, $B\in K_0(R)^{qnil}$. There exists $B' =
\begin{bmatrix}x'&y'\\z'&t'\end{bmatrix}\in K_0(R)^{qnil}$ such
that $BA = AB'$. It implies $ba = ax' = at'$, $ay' = 0$ and $az' =
0$. Hence $x' = t'$ and $y' = z' = 0$. Hence $x'\in R^{qnil}$ by
Proposition \ref{nnil}.
\end{proof}

%\begin{thm} Let $R$ be a field. Then the following hold.
%\begin{enumerate}
%\item[(1)]  \item[(2)]

%\end{enumerate}
%\end{thm}
%\begin{proof} (1)
%\end{proof}

\end{document}